\documentclass[12pt]{article}
\usepackage{amsmath,fullpage,amsthm,amssymb,xcolor,graphicx}
\newtheorem{prop}{Proposition}[section]
\newtheorem{theorem}{Theorem}[section]
\newtheorem{corollary}{Corollary}[section]

\newtheorem{definition}{Definition}[section]
\newtheorem{remark}{Remark}[section]
\newtheorem{example}{Example}[section]

\DeclareMathOperator{\spec}{spec}

\renewcommand{\Re}{\operatorname{Re}}
\renewcommand{\det}{\operatorname{det}}
\title{On the spectrum of complex unit gain graphs}
\author{Aniruddha Samanta \thanks{Department of Mathematics, Indian Institute of Technology Kharagpur, Kharagpur 721302, India. Email: aniruddha.sam@gmail.com}\  \and M. Rajesh Kannan\thanks{Corresponing author}~\thanks{Department of Mathematics, Indian Institute of Technology Hyderabad, Hyderabad 502285, India. Email: rajeshkannan@math.iith.ac.in, rajeshkannan1.m@gmail.com }}

\date{\today}
\begin{document}
	\maketitle
	\baselineskip=0.25in

\begin{abstract}
	
	A complex unit gain graph ($\mathbb{T}$-gain graph) $ \Phi=(G, \varphi) $ is a graph where the gain function $ \varphi $ assigns a unit complex number to each orientation of an edge of $ G $, and its inverse is assigned to the opposite orientation. The adjacency matrix $ A(\Phi) $ of $ \Phi $ is defined canonically. In this article, first  we study  cospectrality of the adjacency matrices of various $ \mathbb{T} $-gain graphs defined on the same underlying graph. Let $ \rho(\Phi) $ and $ \lambda_1(\Phi) $ be the spectral radius and largest eigenvalue of $ A(\Phi) $, respectively. A graph $ X $ which contains both directed and undirected edges is known as a mixed graph. Adjacency matrices of mixed graphs are  particular cases of adjacency matrices of $\mathbb{T}$-gain graphs. For any mixed graph $X$, the following holds:  $ \lambda_1(X)\leq \rho(X)\leq 3\lambda_1(X) $. We construct examples to show this inequality need not be true for arbitrary $\mathbb{T}$-gain graphs. We construct classes of gain graphs for which the above inequality holds.  We consider  new classes of Hermitian matrices $ H_k(X) $, $k = 1,2,\dots$, associated with a mixed graph $ X $. Finally we esatablish that $ \rho(H_k(X))\leq \Delta $, where $ \Delta $ is the largest vertex degree of $ X $, and characterize the structure of $ X $ for which the equality holds. As a consequence, two known results about the spectral radius of adjacency matrices of mixed graphs are deduced.
\end{abstract}



  \section{Introduction}

The study of matrices and their eigenvalues associated with graphs has evolved over the past few decades. There has been a growing interest among researchers to study the adjacency, Laplacian, and normalized Laplacian matrices associated with undirected graphs. If the graph is undirected, then all the above-mentioned matrices are symmetric.     On the other hand, in the case of the directed graphs and the mixed graphs (graphs containing both directed and undirected edges), if we associate the adjacency matrices canonically, then we may get non-symmetric matrices, which are hard to work with. To overcome this drawback, recently researchers came up with various interesting classes of adjacency matrices associated with graphs viz., weighted adjacency matrices, Hermitian adjacency matrices of graphs, and gain adjacency matrices. All these matrices are Hermitian and the entries are complex numbers.  All these matrices associated with the given directed/mixed graphs reveal several properties of the underlying graph.

The main objective of this manuscript is to study the properties of the graphs with the aid of the spectrum of the gain adjacency matrices associated with them. Some of the major contributions of this manuscript include the following:  invariance of the spectrum of various gain adjacency matrices associated with a graph; establishing bounds for the spectral radius of gain adjacency matrix in terms of the largest eigenvalue;  deriving bound for the spectral radius of gain adjacency matrices in terms of the maximum vertex degree of the underlying simple graph, and characterizing the sharpness of this bound (for some cases).  The bound established for the spectral radius of gain adjacency matrices in terms of the largest eigenvalues holds for all gain adjacency matrices with nonnegative real part (also for the gain adjacency matrices switching equivalent to a gain adjacency matrix with nonnegative real part). This leads to the question that which gain graphs are switching equivalent to a gain graph with nonnegative real parts.  It seems this is a hard problem to solve. Nevertheless, we provide a large collection of graphs for which this property holds. Note that, some of these results are known for the Hermitian adjacency matrices of graphs. But the proofs presented here are different from that of the Hermitian adjacency matrices of graphs. In fact, in some cases, our proofs simplify the known proofs and extend for the more general case viz., for the gain adjacency matrices.

For a given group $\mathfrak{G}$, a $\mathfrak{G}$-gain graph is a graph $G$ with each orientation of an edge of $G$ is assigned  an element $g \in \mathfrak{G}$ (called a gain of the oriented edge)  and whose  inverse $g^{-1}$ is assigned to the opposite orientation of the edge. For more details about the notion of $\mathfrak{G}$-gain graphs, we refer to  \cite{Zaslav-gain-balance, Zas4}. Let $\mathbb{T}=\{z\in\mathbb{C}:|z|=1\} $ be the multiplicative group of unit complex numbers. In \cite{reff1}, the notion of $\mathbb{T}$-gain graphs has been studied. If $G$ is a graph with the vertex set $ V(G)=\{v_{1}, v_{2}, \dots, v_{n} \} $ and each orientation of its edges having some gain from $\mathbb{T}$, then the associated  $\mathbb{T}$-gain adjacency matrix is an  $n \times n$ matrix defined  as follows: $(s,t)th$ entry of the matrix is the gain of the edge starts from the vertex $v_{s}$ and ends at the vertex $v_{t}$, and is zero if there is no edge between the vertices $v_{s}$ and $v_{t}$.

One of the motivations to consider the notion of $\mathbb{T}$-gain adjacency matrices is to unify the various known notions of adjacency matrices of simple graphs, digraphs and mixed graphs. The notion of \emph{signed graph} is a particular case of $\mathbb{T}$-gain graph, where the gains are taken from $\{\pm 1\}$. In \cite{Bap,Kat}, the authors considered complex weighted graphs where the weights are from $\{\pm1,\pm i\}$, which is again a particular case of gain graphs. Hermitian adjacency matrix of a mixed graph, which was introduced in \cite{Bojan, Lie}, is a particular case with gains $\{1, \pm i\}$. Even though $\{1, \pm i\}$ is not a group, we need only inverse closed sets here.  Recently, in \cite{mohar-new} the notion of Hermitian adjacency matrices of second kind for mixed graphs is introduced.

Particular cases of the notion of adjacency matrices of  $\mathbb{T}$-gain graphs were considered  with different gains  in the literature \cite{Bap,Kat}.  In \cite{Ger}, the authors studied some properties of the characteristic polynomial of the $\mathbb{T}$-gain graphs. For some interesting spectral properties of $\mathbb{T}$-gain graphs, we refer to \cite{joac-gain,  reff1, Reff2016, Ger,  Zaslav, Zas2}.

A \emph{directed graph (or digraph)} $ X $ is an ordered pair $ (V(X), E(X)) $, where $ V(X)=\{ v_{1}, v_{2}, \dots,v_{n}\} $ is the vertex set and $ E(X) $ is the directed edge set. A directed edge from the vertex $ v_{s} $ to the vertex $ v_{t} $ is denoted by $ \overrightarrow{e_{st}} $. If $ \overrightarrow{e_{st}} \in E(X)$ and $  \overrightarrow{e_{ts}}\in E(X)$, then the pair $ \{v_{s},v_{t}\} $ is called a \emph{digon} of $ X $. The  underlying graph of $ X $ is a simple undirected graph obtained from $ X $ by replacing a directed edge by an undirected edge and it is denoted by $ \Gamma(X) $.  The  \textit{Hermitian adjacency matrix} \cite{Bojan, Lie} of a digraph $ X $ is  denoted by $H(X)$ and is defined as follows:
$$\mbox{$ (s,t){th}$ entry of }H(X)=h_{st}=\begin{cases}
	1& \text{if  } \mbox{$\overrightarrow{e_{st}}\in E(X) $ \text{and} $\overrightarrow{e_{ts}} \in E(X)$},\\
	i& \text{if  } \mbox{$\overrightarrow{e_{st}} \in E(X)$ \text{and} $\overrightarrow{e_{ts}} \notin E(X)$},\\
	-i& \text{if  } \mbox{$\overrightarrow{e_{st}} \notin E(X)$ \text{and} $\overrightarrow{e_{ts}} \in E(X)$},\\
	0&\text{otherwise.}\end{cases}$$
The Hermitian adjacency matrix can be thought of as the adjacency matrix of a $\mathbb{T}$-gain graph where the gains are from $ \{1, \pm i\}$.

Let $\mathbb{T}_{G} $ denote the collection of all  $\mathbb{T}$-gain adjacency matrices on a graph $G$. The first main objective of this article is to study the relationship between  the spectral properties of various  $\mathbb{T}$-gain adjacency matrices associated with $G$. In section \ref{co-spec},  we show that either the spectrum or the spectral radii of the $\mathbb{T}$-gain adjacency matrices on a graph $G$ are invariant under all $\mathbb{T}$-gains if and only if  $G$ is a tree [cf, Theorem \ref{spec-equi-tree}]. Then, we identify a class of connected graphs such that for every graph $G$ in this class, the co-spectrality of adjacency matrices in $\mathbb{T}_G$ is determined by the gains of the fundamental cycles.

The second objective of this article is to introduce a class of connected graphs, called $ \mathcal{F} $-graphs,  such that for each $G$ in this collection, the matrices in $\mathbb{T}_G$ have nonnegative real part  up to the diagonal unitary similarity.
We extend some of the bounds  for the spectral radius of the gain adjacency matrices of $ \mathcal{F} $-graphs , which are known for Hermitian adjacency matrices.

The spectral radius of the adjacency matrix $ A(G) $ of a graph $ G $ is the largest eigenvalue of $A(G)$.  If we consider directed graph or weighted graph with complex numbers as edge weights, then the spectral radius need not be an eigenvalue of the associated adjacency matrix. For the Hermitian adjacency matrix of a digraph, surprisingly, in \cite[Theorem 5.6]{Bojan}, the authors established the following bounds for the  spectral radius, denoted by $\rho(H(X))$, of the Hermitian adjacency matrix $ H(X) $  in terms of its largest eigenvalue $\lambda_{1}(H(X))$.

\begin{theorem}[{\cite[Theorem 5.6]{Bojan}}]\label{Theo.6}
	For every digraph $X$, $ \lambda_{1}(H(X))\leq \rho(H(X))\leq3\lambda_{1}(H(X))$.
\end{theorem}

The third objective of this article is to extend the above result for $\mathbb{T}$-gain adjacency matrices. Interestingly, the above bound need not be true for arbitrary gain graphs. Infact, the difference between the spectral radius and the largest eigenvalue of gain graphs can be arbitrarily large. For any connected graph, we show that if  the associated $\mathbb{T}$-gain graph has nonnegative real part, then  the above bounds hold true [cf, Theorem \ref{Theo.5.1}].

The spectral radius of the adjacency matrix of a graph is bounded above by the maximum vertex degree of the underlying graph. In \cite{Bojan}, this bound was extended to Hermitian adjacency matrices. Also the sharpness of the bound has been characterized  \cite[Theorem 5.1] {Bojan}. Our final objective is to introduce the notion of $ k $-generalized Hermitian adjacency matrix, $ H_{k}(X) $ of a mixed graph $ X $, which extends the notion of Hermitian adjacency matrix of $X$.
For this new class of matrices, we characterize the structure of a mixed graph $ X $ for which the spectral radius $ \rho(H_{k}(X))$  equals to the largest vertex degree of $\Gamma(X)$, and particularly $ k=1 $ gives an alternative  proof of some of the known results for Hermitian adjacency matrices.

\section{Notations, definitions and known results}

Let $ G=(V(G), E(G)) $ be a simple undirected graph with the vertex set $ V(G)=\{v_{1}, v_{2},\dots, v_{n} \} $ and the edge set $ E(G) $. If the vertex $ v_{i} $ is adjacent to the vertex $ v_{j} $, then we write $ v_{i} \sim v_{j}$. The undirected edge between the vertices $ v_{i} $ and $ v_{j} $, if exists,  is denoted by  $ e_{ij} $. The \textit{adjacency matrix} of a simple graph $ G $, denoted by $ A(G) $,  is the symmetric  $ n \times n $  matrix whose  $ (i,j)th $ entry  is defined by $ a_{ij}=1 $ if $ v_{i} \sim v_{j} $, and  $ a_{ij} = 0 $ otherwise.

A digraph is said to be an \emph{oriented graph} if it has no digons. A \textit{ mixed graph } is a graph which may contain both directed and undirected edges. When we consider Hermitian adjacency matrix of a mixed graph,   the undirected edges are treated as digons. From this point of view, digraphs and mixed graphs are equivalent.

An \textit{path} in an undirected tree $ T $ between the vertices $ v_{i} $ and $ v_{j} $ is denoted by  $ v_{i}Tv_{j} $. If $S$ is a directed tree, then  $\overrightarrow{v_{i}Sv_{j}} $ denote the \textit{directed path} in  $ S $ from the vertex $ v_{i} $ to the vertex $ v_{j}$.
A \emph{rooted tree} is a tree in which one vertex has been fixed as the root. Let $T$ be a rooted tree with the vertex $v_{r}$ as the root. The vertex set of the rooted tree $T$ admits a canonical partial ordering on it as follows:  $ v_{x} \leq v_{y}$ if the vertex $ v_{x} $ lies in the path $ v_{r}Tv_{y} $. This partial order is called the \textit{tree-order} on $ V(T) $ associated with the rooted tree $ T $ with root vertex $ v_{r} $ \cite{Diestel}. A rooted spanning tree $ T $ of a connected graph $ G $ is said to be a \textit{normal spanning tree} if  any two adjacent vertices of $G $ are comparable with respect to the tree ordering. Whenever we  consider $ T $ as a normal spanning tree of a connected graph $ G $, we assume that $ T $ is a rooted tree with some vertex as its root.	If  $ G $ is a connected graph, then $ G $ has a normal spanning tree with any specified vertex as its root \cite[Proposition 1.5.6]{Diestel} . Let $ G $ be a connected graph with spanning tree $ T $. Then for each edge $ e \in  E(G)\setminus E(T) $, adding the edge $e$ to $T$ creates a unique cycle in $T \cup \{e\}$. This cycle  is called a \textit{ fundamental cycle} of $ G $.

In this article, we call $ \overrightarrow{e_{st}} $  a directed edge in the context of digraphs, and we call the same as an oriented edge in the context of gain graphs.
For any simple graph $ G $, each undirected edge $ e_{st} \in E(G) $ is associated with a  pair of oriented edges, namely $ \overrightarrow{e_{st}} $ and $ \overrightarrow{e_{ts}} $. Set of all such oriented edges of a simple graph $ G $ is known as \textit{the oriented edge set} of $ G $, and is denoted by $ \overrightarrow{E}(G) $. A \textit{ $\mathbb{T}$-gain graph (or complex unit gain graph)} on a simple graph $ G $ is a triplet $\Phi=(G,\mathbb{T},\varphi)$ such that  the map (\emph{the gain function}) $\varphi:\overrightarrow{E}(G)\rightarrow\mathbb{T}$  satisfies $\varphi(\overrightarrow{e_{st}})=\varphi(\overrightarrow{e_{ts}})^{-1}$. That is, for an oriented edge $\overrightarrow{e_{st}}$, if we assign a value $g$ (the \textit{gain} of the edge $ \overrightarrow{e_{st}}$) from $\mathbb{T}$, then assign $g^{-1 }$ to the oriented edge $\overrightarrow{e_{ts}}$. For simplicity, we use $\Phi=(G,\varphi)$ to denote a $\mathbb{T}$-gain graph instead of $\Phi=(G,\mathbb{T},\varphi)$. We call $ \varphi $  a $ \mathbb{T}$-gain on $ G $ if $ \Phi=(G, \varphi) $ is a $ \mathbb{T} $-gain graph on $ G $. In \cite{reff1},  the author studied the notion of the adjacency matrix $A(\Phi)=(a_{st})_{n\times n}$ of a $\mathbb{T}$-gain graph $\Phi$. The entries of $A(\Phi)$ are given by
$$a_{st}=\begin{cases}
	\varphi(\overrightarrow{e_{st}})&\text{if } \mbox{$v_s\sim v_t$},\\
	0&\text{otherwise.}\end{cases}$$
It is clear that the matrix $A(\Phi)$ is Hermitian, and hence its eigenvalues are real. When $\varphi(\overrightarrow{e_{st}})=1$ for all $\overrightarrow{e_{st}}$, then $A(\Phi)=A(G)$. Thus we can consider $G$ as a $\mathbb{T}$-gain graph and we write this by $(G,1)$.

For a square matrix $ B $ with complex entries, $\spec(B)$ and $\rho(B)$ denote the spectrum and the spectral radius of $ B $, respectively.

A cycle is called {\emph{a directed cycle}} if its all edges are oriented in the same direction. A directed cycle obtained from  a cycle $C$ is denoted by $ \overrightarrow{C} $. For any cycle, there are only two directed cycles associated with it. Let  $ \overrightarrow{C} \equiv v_{1} \overrightarrow{e_{12}}v_{2}\overrightarrow{e_{23}} \dots v_{k}\overrightarrow{e_{k1}}v_{1} $ be a directed cycle in a $ \mathbb{T}$-gain graph $ \Phi=(G, \varphi) $, then the gain of this cycle, denoted by $\varphi(\overrightarrow{C})$, is defined as $\varphi(\overrightarrow{C}) := \varphi(\overrightarrow{e_{12}})\varphi(\overrightarrow{e_{23}}) \cdots \varphi(\overrightarrow{e_{k1}}) $. If $\varphi(\overrightarrow{C})$ =1, we call the underlying cycle $ C $  neutral in $ \Phi $. A $\mathbb{T} $-gain graph $ \Phi $ is said to be  \textit{balanced} if all the cycles in $ G $ are neutral in $ \Phi $. A \textit{potential function} for $\varphi$ is a function $\psi:V\rightarrow\mathbb{T}$, such that for each edge $e_{ij}$, $\varphi(e_{ij})=\psi(v_i)^{-1}\psi(v_j).$ Given a gain graph $\Phi$, define $-\Phi = (G, -\varphi).$

Any function from the vertex set of $G$ to the  complex unit circle $\mathbb{T}$ is called a \textit{switching function}. Two $\mathbb{T}$-gain graphs $\Phi_1=(G,\varphi_1)$ and $\Phi_2=(G,\varphi_2)$ are said to be \textit{switching equivalent}, denoted by $\Phi_1\sim\Phi_2$, if there is a switching function $\zeta:V(G)\rightarrow\mathbb{T}$ such that $$\varphi_2(\overrightarrow{e_{st}})=\zeta(v_s)^{-1}\varphi_1(\overrightarrow{e_{st}})\zeta(v_t).$$

The switching equivalence of two $\mathbb{T}$-gain graphs can be defined in the following equivalent way: Two $\mathbb{T}$-gain graphs $\Phi_1=(G,\varphi_1)$ and $\Phi_2=(G,\varphi_2)$ are switching equivalent, if there exists a diagonal matrix $D_\zeta$ with diagonal  entries from $\mathbb{T}$ such that
$$A(\Phi_2)=D_\zeta^{-1}A(\Phi_1)D_\zeta.$$

The following result gives an upper bound for the spectral radius of $A(\Phi)$ in terms of the  maximum vertex degree $\Delta$ of $G$.
\begin{theorem} [{\cite[Theorem 4.3]{reff1}}] \label{Theo.12}
	Let $ \varphi $ be a $ \mathbb{T} $-gain on a graph $ G $. Then $ \rho(A(\Phi)) \leq \Delta. $
\end{theorem}

Next, we collect a couple of results related to the spectrum and the spectral radius of adjacency matrices of $\mathbb{T}$-gain graphs.

%
%


%

\begin{theorem}[{\cite[Theorem 4.4]{Our-gain1}}]\label{Theo.9}
	Let $ \varphi $ be a $ \mathbb{T} $-gain on a connected graph $ G $, then $\rho(A(\Phi))=\rho(A(G))$ if and only if either $\Phi$ or $ -\Phi $ is balanced.
\end{theorem}

\begin{theorem}[{\cite[Theorem 4.6]{Our-gain1}}] \label{Theo.7}
	Let $ \varphi $ be a $ \mathbb{T} $-gain on a connected graph $ G $. Then, $ \spec(A(\Phi)) =\spec(A(G))$ if and only if $ \Phi $ is balanced.
\end{theorem}

The \textit{characteristic polynomial} of an $n \times n$ matrix $ M $  is defined as $ \det(M-xI) $, where $ I $ is the  $n \times n$ identity matrix.
A graph $G$ is called an \textit{elementary graph}, if each of its component is either an edge or a cycle. Let $\mathcal{H}(G)$ denote the collection of all spanning elementary subgraphs of a graph $G$. For any $H \in\mathcal{H}(G)$, let $\mathcal{C}(H)$ denote the collection of all cycles in $H$.
\begin{theorem}[{ \cite[Corollary 3.1]{Our-gain1}}]\label{Theo4}
	Let $\Phi$ be any $\mathbb{T}$-gain graph with the underlying graph $G$. Let $P_\Phi(x)=x^n+a_1x^{n-1}+\cdots+a_n$ be the characteristic polynomial of $A(\Phi)$. Then
	$$a_i=\sum_{H\in\mathcal{H}_i(G)}(-1)^{p(H)}2^{c(H)}\prod_{C\in \mathcal{C}(H)}\Re(\varphi(C)),$$
	where $\mathcal{H}_i(G)$ is the set of all elementary subgraphs of $G$ with $i$ vertices. $ p(H) $ and $ c(H) $ are the number of components and the number of cycles in $ H $, respectively.  $\Re(\varphi(C)) $ is the real part of the gain of a directed cycle $ \overrightarrow{C}$.
\end{theorem}

Let $ A$ be an $n \times n$ Hermitian matrix and let the eigenvalues of $ A $ be ordered as $\lambda_{\min}=\lambda_n\leq\lambda_{n-1}\leq \ldots \leq\lambda_1=\lambda_{\max}$. Then $ \lambda_{min} \leq x^{*}Ax \leq \lambda_{\max} $ for any  vector $ x \in \mathbb{C} ^{n}$ with $x^{*}x=1$, with equality in the right-hand (respectively, left-hand) inequality if and only if $ Ax=\lambda_{\max} x$ (respectively, $ Ax=\lambda_{\min}x $). Moreover,

\begin{center}
	$ \lambda_{\max}= \max_{x \neq 0}\frac{x^{*}Ax}{x^{*}x} $ and $ \lambda_{\min}= \min_{x \neq 0}\frac{x^{*}Ax}{x^{*}x} $.
\end{center}

\vspace*{8pt}
The \textit{numerical range } of an $n \times n$ complex matrix $A$ is a subset of the complex numbers $ \mathbb{C} $, defined as follows:
\begin{center}
	$ W(A):=\{\langle Ax, x \rangle: x \in \mathbb{C}^n, \langle x, x \rangle=1  \} $.
\end{center}
For any $n \times n $ complex matrix $A$, the numerical range $W(A)$ is a convex set \cite{nrange}.
\begin{theorem}[{\cite{nrange}}] \label{Theo.5}
	If $A$ is an $n \times n$ Hermitian matrix, then $ W(A)$ is an interval $ [m,M] $, where  $ m=\lambda_{min}(A), M=\lambda_{max}(A) $. Moreover, $ \rho(A) =   \max \{|m|,|M|\} $.
	
\end{theorem}

\section{Cospectral $ \mathbb{T} $-gain graphs}\label{co-spec}

For any simple graph $ G $, the collection of all $ \mathbb{T} $-gain graphs associated with $ G $ is denoted by $\mathcal{T}_{G}$. Define $ \mathbb{T}_{G}:=\{A(\Phi): \Phi \in \mathcal{T}_{G} \} $. The spectrum of $ A(\Phi) $, denoted by $\spec (A(\Phi))$ (or simply $\spec(\Phi)$), is the  \textit{spectrum} of $ \Phi $. The spectral radius of $A(\Phi)$, denoted by $\rho(A(\Phi))$ (or simply $\rho(\Phi)$), is the spectral radius of $\Phi$. Two gain graphs  $ \Phi_{1}=(G, \varphi_{1})$ and $ \Phi_{2}=(G, \varphi_{2}) $ with the same spectrum are called \textit{cospectral} $ \mathbb{T}$- gain graphs.

A considerable amount of literature is available on constructing cospectral simple graphs. From the $ \mathbb{T}$-gain point of view, we may ask the similar question about the class of cospectral gain graphs on $ G $.
\begin{figure}[hb!]
	\begin{center}
		\includegraphics[scale= 0.85]{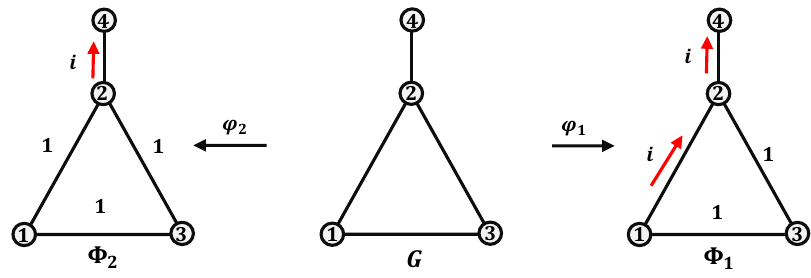}\\
		\caption{ Two  $ \mathbb{T} $-gain graphs $ \Phi_{1} $ and $ \Phi_{2} $ on the same underlying graph $ G $ with different spectrum. } \label{fig1}
	\end{center}
\end{figure}
For the gain graphs given in Figure \ref{fig1}, we have $ \spec (A(\Phi_{1}))=\{-1.9318, -0.5176, 0.5176, 1.9318 \} $ and $ \spec (A(\Phi_{2}))   =\{ 1.4811, -1, 0.3111, 2.1701\} $.
Thus, in spite of having the same underlying graph $ G $, $ A(\Phi_{1}) $ and $A(\Phi_{2}) $ have different spectrum. So on this basis it is natural to ask the following question: For which graph $G$, all the gain graphs defined on $G$ are cospectral. We answer to this in the next theorem. 
\begin{theorem}\label{spec-equi-tree}
	Let $ \Phi=(G,\varphi)$ be a $ \mathbb{T}$-gain graph. Then the following  are equivalent:
	\begin{enumerate}
		\item[(1)] $G$ is a tree.
		\item[(2)] Any two $ \mathbb{T} $-gain graphs on $ G $ are cospectral.
		\item[(3)]  The spectral radius of any two $ \mathbb{T} $-gain graphs on $ G $ are the same.
	\end{enumerate}
\end{theorem}

\begin{proof}
	(1) $\Rightarrow$ (2): Since $ G $ is a tree, all $ \mathbb{T} $-gain graphs on $ G $ are balanced. Therefore, by Theorem \ref{Theo.7}, $ \spec(G) = \spec( \Phi)$, for any $ \Phi $. Hence, any two $ \mathbb{T} $-gain graphs on $ G $ are cospectral. \\
	(2) $\Rightarrow$ (3): Easy to verify. \\
	(3) $\Rightarrow$ (1): By Theorem \ref{Theo.9}, we know that $ \rho(G)=\rho(\Phi) $ if and only if either $ \Phi $ or $ -\Phi $ is balanced. Therefore, $ \rho(\Phi) $ is invariant, which implies that either $ \Phi $ or $ -\Phi $ is balanced for any $ \varphi $.  This is possible only when $ G $ is a tree. Otherwise, if there is a cycle in $G$, then  we can construct a $ \mathbb{T} $-gain graph  $\Phi$  on $ G $ so that,  neither $ \Phi $ nor $ -\Phi $ is balanced.
\end{proof}

Next, we define the notion of a suitably oriented graph. Using this notion, we construct uncountably many matrices in  $ \mathbb{T}_{G} $ having the same spectrum.

\begin{definition}{\em 
		Let $ G $ be a connected graph with vertex set $ V(G)=\{ v_1, v_2, \cdots, v_n\}$ and $ T $ be a normal spanning tree of $ G $ with root vertex $v_r $. A \emph{suitably oriented graph} of $ G $  associated with $ T $ is an oriented graph, denoted by $ \vec{G}_T $, defined by assigning orientation on each edge of $ G $ as follows:
		
		\begin{enumerate}
			\item[(i)] Tree edges are oriented away from the root.
			
			\item[(ii)] Non tree edges are oriented forwards the root. 
	\end{enumerate}}
\end{definition}

A cycle is  \emph{oriented cycle} if each edge of the cycle is oriented in a fixed direction. Note that the oriented cycles are nothing but the directed cycles. By the definition of suitably oriented graph, observe that all the fundamental cycles in $ \vec{G}_{T} $  with respect to $ T $ are oriented, and they are denoted by $ \vec{C}^T_1, \vec{C}^T_2, \dots, \vec{C}^T_{m-n+1} $. In this section, we simply write the oriented fundamental cycles as $ \vec{C}_1, \vec{C}_2, \dots, \vec{C}_{m-n+1} $, when there is no confusion about $ T $.

Let $ \Phi=(G, \varphi) $ be a $ \mathbb{T} $-gain graph on a connected graph $ G $. Let $\vec{e_1}, \vec{e_2}, \dots, \vec{e_m} $ be the oriented edges of $ \vec{G}_T $. Let $ \varphi(\vec{e}_{k})=e^{i\theta_k} $, where $ \theta_k \in [0, 2\pi)$, $ k=1,2, \cdots, m $. Let us consider an $ m $-vector denoted by $ \theta_T(\Phi):=(\theta_{1}, \theta_{2}, \dots, \theta_{m})$. Then, a $ \mathbb{T} $-gain graph $ \Phi=(G, \varphi) $ with a normal spanning tree $T$ can be uniquely identified with $ \theta_T(\Phi)$.

\begin{theorem}\label{Th-3.1.1}
	Let $ \Phi$ and $ \Psi $ be two $ \mathbb{T} $-gain graphs on a connected graph $ G $ with $ m $ edges and $ n $ vertices. If $ \varphi(\vec{C}_k)=\psi(\vec{C}_k)$, for $ k=1,2, \cdots, (m-n+1) $, where $ C_k$'s are the fundamental cycles with respect to some normal spanning tree $T$ of G, then $ \spec(\Phi) =\spec(\Psi)$.
\end{theorem}

\begin{proof}
Let us consider the suitable orientation $ \vec{G}_T $ of $ G $ induced by $T$. Let $ Q $ be the incidence matrix of $ G $ corresponding to the suitable orientation. Suppose $ E( \vec{G}_T)=\{\vec{e}_1, \vec{e}_2, \dots, \vec{e}_m\} $ is the oriented edge set of $ \vec{G}_T $. Let $ C $ be a cycle in $ G $, and $ \vec{C} $ be the orientated cycle associated with $ C $ in  $ \vec{G}_T $. Let $ I(\vec{C}) $ be the $ m \times 1 $ incidence vector $\vec{C}$ where the components are indexed by  the elements of $ E(\vec{G}_T)$. Then the $ j$-th component of $ I(\vec{C}) $ is $ 1 $ if $ \vec{e}_j $ is an edge of $ \vec{C} $, and zero otherwise. Let $ \vec{C}_1, \vec{C}_2, \dots, \vec{C}_{m+n-1} $ be the oriented fundamental cycles of $ \vec{G}_T $. Then it is clear that $ \{ I(\vec{C}_k): k=1,2, \dots, (m-n+1)\} $ forms a basis of null space of $ Q $. Also $ I(\vec{C}) $ belongs to the null space of $ Q $. Thus $I(\vec{C})=\sum\limits_{k=1}^{m-n+1}c_{k}I(\vec{C}_k)$, for some real numbers $c_k$. Now, $ \varphi(\vec{C})=e^{i\langle  I(\vec{C}), \theta_T(\Phi)\rangle}$, where $ \langle ~,\rangle $ denotes the usual inner product. Then
	\begin{align*}
		\varphi(\vec{C}) &=e^{i\langle  I(\vec{C)}, \theta_T(\Phi)\rangle} \\
		&=e^{i \langle c_{1}I(\vec{C}_1)+c_{2}I(\vec{C}_2)+ \dots+ c_{m-n+1}I(\vec{C}_{m-n+1}), \theta_T(\Phi) \rangle }\\
		&=e^{{ c_{1} i \langle I(\vec{C}_1) , \theta_T(\Phi) \rangle }+{c_{2} i \langle I(\vec{C}_2) , \theta_T(\Phi) \rangle }+ \dots +{c_{m-n+1} i\langle  I(\vec{C}_{m-n+1}) , \theta_T(\Phi)\rangle }}\\
		&=\left({e^{i \langle  I(\vec{C}_1) , \theta_T(\Phi) \rangle  }} \right) ^{c_{1}} \left( {e^{i  \langle I(\vec{C}_2) , \theta_T(\Phi) \rangle }}\right)^{c_{2}}\dots \left( {e^{i \langle  I(\vec{C}_{m-n+1}) , \theta_T(\Phi) \rangle  }}\right) ^{c_{m-n+1}}
	\end{align*}
	
	Since $ \psi(\vec{C}_k)={e^{i \langle  I(\vec{C}_{k}) , \theta_T(\Psi) \rangle  }}={e^{i \langle  I(\vec{C}_{k}) , \theta_T(\Phi) \rangle  }}=\varphi(\vec{C}_k)$, for all $ k $, so from the above expression, $ \varphi(\vec{C})= \psi(\vec{C})$.
	
	Let $P_\Phi(x)=x^n+a_1x^{n-1}+\cdots+a_n$ be the characteristic polynomial of $ \Phi=(G, \varphi) $. Then, by Theorem \ref{Theo4},
	$$a_j=\sum_{H\in\mathcal{H}_j(G)}(-1)^{p(H)}2^{c(H)}\prod_{C\in \mathcal{C}(H)}\Re(\varphi(C)), ~~ j=1,2, \dots, n.$$
	Since  $ \varphi(\vec{C})= \psi(\vec{C})$, we have $ \Re( \varphi(C))=\Re(\psi(C))$ for any cycle $ C $, and hence $ P_\Phi(x)=P_\Psi(x) $. 
\end{proof}

\begin{corollary}\label{Cor-3.1.1}
	Let $ \Phi$ and $ \Psi $ be two connected $\mathbb{T}$-gain graphs on $ G $ with $ m $ edges and $ n $ vertices. Let $ C_1, C_2, \cdots, C_{m-n+1} $ be the fundamental cycles of $ G $ associated with a normal spanning tree $ T $. Then $ \Phi \sim \Psi $ if and only if $ \varphi(\vec{C_j})=\psi(\vec{C_j})  $, for $ j=1,2, \cdots, (m-n+1) $.
\end{corollary}


Let $ \Phi_1=(G, \varphi_1) $ and $ \Phi_2=(G, \varphi_1)  $ be two $ \mathbb{T} $-gain graphs. If $ \Re(\varphi_1(C))=\Re(\varphi_2(C)) $, for all cycles $ C $, then, by Theorem \ref{Theo4}, $ \spec(\Phi_1)=\spec(\Phi_2) $. The converse of this statement need not be true.
\begin{example}
	{\rm    Let us consider two $ \mathbb{T} $-gain graphs $ \Phi_1$ and $ \Phi_2$ given in  Figure \ref{fig3.1}. Here $ \spec(\Phi_1)=\spec(\Phi_2)=\{ -2.37, -1.41, -0.59, 0, 0.59, 1.41, 2.37 \} $ but $ \Re(\varphi_1(C_{i})) \ne \Re(\varphi_2(C_{i})) $, $ i=1,2$.  \begin{figure}[!htb]
			\begin{center}
				\includegraphics[scale= 0.77]{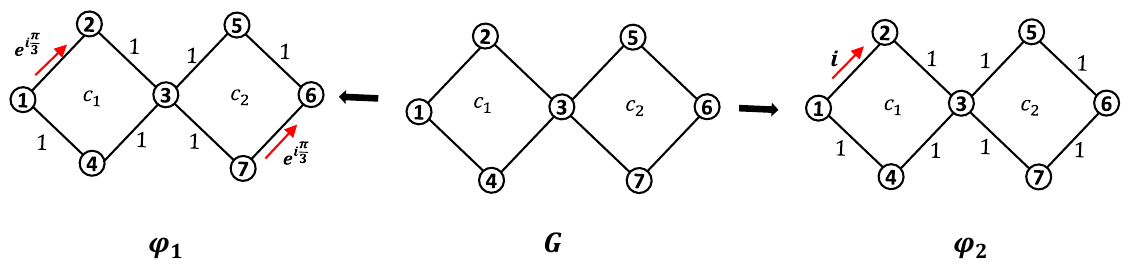}\\
				\caption{Two cospectral graphs with different cycle gains}\label{fig3.1}
			\end{center}
	\end{figure}}
\end{example}

Next, we identify a class of graphs for which $ \spec(\Phi_1)$ and $\spec(\Phi_2)$ are the same if and only if $ \Re(\varphi_{1}(C))= \Re(\varphi_{2}(C)) $ holds for all cycles $ C $ in $ G $.

Let $ \mathcal{S}_{n} $ denote the collection of all connected graphs $ G $ of $ n $ vertices such that for each $ k $, $ 3\leq k\leq n $, $ G $ have at most one cycle of length $ k $. The graphs $ G_1 $ and $ G_2 $ in Figure \ref{fig.4} are elements of $ \mathcal{S}_{10} $. 

\begin{figure}[!htb]
	\begin{center}
		\includegraphics[scale= 0.80]{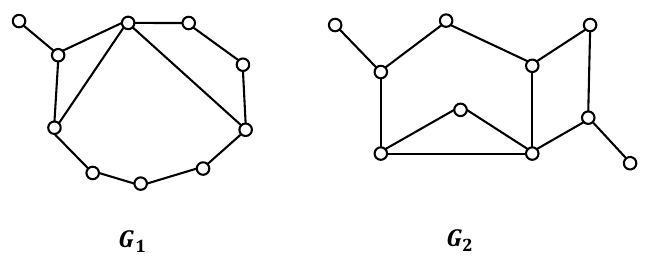}\\
		\caption{Both the graphs $ G_{1} $ and $ G_{2} $ are in $ \mathcal{S}_{10} $} \label{fig.4}
	\end{center}
\end{figure}

\begin{theorem}
	Let $ \Phi_{1} $ and $ \Phi_{2} $ be two $ \mathbb{T} $-gains on $ G \in \mathcal{S}_{n} $. Then $ \Phi_1 $ and $ \Phi_2 $ are cospectral if and only if $ \Re(\varphi_{1}(C))=\Re(\varphi_{2}(C)) $ for all the cycles $ C$  in G.
\end{theorem}

\begin{proof}
	Let  us assume that $ \spec(\Phi_{1})=\spec(\Phi_{2}) $. Let
	$ P_{\Phi_{1}}(x)=x^{n}+a_{1}x^{n-1}+ \dots + a_{n} $ and $  P_{\Phi_{2}}(x) =x^{n}+b_{1}x^{n-1}+ \dots + b_{n} $ be the characteristic polynomials of $A(\Phi_{1}) $ and $ A(\Phi_{2}) $, respectively. Therefore, $ a_{i}=b_{i} $ for all $ i$. From Theorem \ref{Theo4},     $$a_i=\sum_{H\in\mathcal{H}_i(G)}(-1)^{p(H)}2^{c(H)}\prod_{C\in \mathcal{C}(H)} \Re (\varphi_{1}(C)), b_i=\sum_{H\in\mathcal{H}_i(G)}(-1)^{p(H)}2^{c(H)}\prod_{C\in \mathcal{C}(H)} \Re(\varphi_{2}(C)) .$$
	Let $ Com{(m_{1}, m_{2},\dots,m_{k})} $ denote the collection of all elementary subgraphs $ H $ with $ k $ components such that the number of vertices in each component are $ m_{1}, m_{2}, \dots , m_{k}$, respectively, where $ m_{1} \geq m_{2} \geq \dots \geq m_{k} $. Let $ m=m_{1}+m_{2}+ \dots+m_{k} $ and $ S_{{(m_{1}, m_{2},\dots,m_{k})}}(\varphi_{1}) $ be the sum of the terms in the expression of $ a_{m} $ associated with the elementary subgraphs in $ Com{(m_{1}, m_{2},\dots,m_{k})} $. It is clear that $ a_{1}=b_{1} $, $ a_{2}=b_{2} $ for any $ \varphi_{1} $ and $ \varphi_{2} $. $ a_{3}=b_{3} $ gives us that $ \Re(\varphi_{1}(C))=\Re(\varphi_{2}(C)),$ where $ C $ denotes a cycle on three vertices if exists. By induction, we may assume that $ a_{j}=b_{j} $ for each $ j<m $ which implies that $ \Re(\varphi_{1}(C_{j}))=\Re(\varphi_{2}(C_{j})) $ for $ j<m $, where $C_{j} $ is the unique cycle of length $ j $ if exists. Now, consider $ a_{m}=b_{m} $, therefore we have
	\begin{equation} \label{eq1}
		S_{(m)}(\varphi_{1})+S_{(m-2,2)}(\varphi_{1})+S_{(m-3,3)} (\varphi_{1})+\cdots + S_{(m_{1}, m_{2},\dots,m_{k})}(\varphi_{1})+ \cdots $$
		$$ = S_{(m)}(\varphi_{2})+S_{(m-2, 2)}(\varphi_{2})+S_{(m-3,3)}(\varphi_{2})+\dots + S_{(m_{1}, m_{2},\dots,m_{k})}(\varphi_{2})+ \cdots
	\end{equation}
	From the above expression, if a representing term  $ S_{(m_{1}, m_{2},\dots,m_{k})}(\varphi_{1}) $ is nonzero, then we have 
	$$ S_{(m_{1}, m_{2},\dots,m_{k})}(\varphi_{1}) = \sum (-1)^{k}2^{l}\Re( \varphi_{1}(C_{m_{1}}))\Re(\varphi_{1}(C_{m_{2}})) \cdots \Re(\varphi_{1}(C_{m_{l}})),$$ $m_{l}>2$, 	 $ m_{l+1}=\dots = m_{k}=2 $ and  the sum runs over all the elementary subgraphs in $ Com{(m_{1}, m_{2},\dots,m_{k})}$.
	
	Therefore, except the first term from both sides of  Equation (\ref{eq1}), for any other term,\break $ S_{(m_{1}, m_{2},\dots,m_{k})}(\varphi_{1})= S_{(m_{1}, m_{2},\dots,m_{k})}(\varphi_{2})$ due to $ m_{1} < m $. Hence from the Equation (\ref{eq1}), we have $ S_{(m)}(\varphi_{1})=S_{(m)}(\varphi_{2})$. That is, $ \Re(\varphi_{1}(C_{m}))=\Re(\varphi_{2}(C_{m})) $. Thus, by induction hypothesis,  $ \Re(\varphi_{1}(C))=\Re(\varphi_{2}(C)) $ for all cycles $ C $  in $ G $. Converse is easy to verify.
\end{proof}

We close this section by posing the following question: Characterize all graphs $G$  for which $ \spec(\Phi_1)$ and $\spec(\Phi_2)$ are the same if and only if $ \Re(\varphi_{1}(C))= \Re(\varphi_{2}(C)) $ holds for all cycles $ C $ in $ G $.

\section{ Spectral radius and the largest eigenvalue}\label{spec_redius1}


For an undirected graph $ G $, by the Perron-Frobenius theorem,  the spectral radius of $A(G)$ is same as the largest eigenvalue of $A(G)$. That is, $ \rho(G)=\lambda_1(G) $. For any mixed graph $ X $, Guo and Mohar \cite{Bojan} established that $ \lambda_{1}(H(X)) \leq \rho(H(X))\leq 3\lambda_{1}(H(X))$, where $ H(X) $ is the Hermitian adjacency matrix of $ X $. Later, in \cite{mohar-new} Mohar extended the bounds for $ N^\alpha(X) $. For the defintion of $N^\alpha(X)$, we refer to \cite{mohar-new}. We note that the above inequality need not be true for  $ \mathbb{T} $-gain graphs. 
\begin{example}\label{Ex4.1}
	Let $ \Phi=(K_n, \varphi) $, where $K_n$ is the complete graph on $n \geq 5$ vertices,  be a $ \mathbb{T} $-gain graph on $ K_n $ such that $ \varphi(\vec{e})=-1$, for any edge $ e $ of $ K_n $. Then $ \lambda_1(\Phi)=1 $ and $ \rho(\Phi)=n-1 > 3\lambda_{1}(\Phi)$. 
\end{example}
 Our main objective in this section is to study the classes of $\mathbb{T}$-gain graphs for which the above  inequality holds. First we show that if the real parts of the gains of a gain graph $\Phi$ are nonnegative, then the above inequality holds (irrespective of the structure of the underlying graph). Then we construct a family of undirected graphs, viz.  $ \mathcal{F} $-graphs, for which the above inequality holds irrespective of the gains defined on the edges. 


\begin{theorem}\label{main4.1}\label{Theo.5.1}
	Let $ \Phi=(G, \varphi)$ be any $ \mathbb{T} $-gain graph on a connected graph $ G $ such that $ \varphi(\overrightarrow{E}(G)) \subseteq \{a+ib \in \mathbb{T}: a \geq 0 \} $. Then $ \lambda_{1}(\Phi)\leq \rho(\Phi)\leq3\lambda_{1}(\Phi)$.
\end{theorem}

\begin{proof}
	Let  $ \lambda_{1} \geq \lambda_{2} \geq \dots\geq \lambda_{n} $ be the eigenvalues of $A(\Phi)$. Let us split the adjacency matrix $ A(\Phi) = P+R $, where $ P = \Re(A( \Phi )) $, the real part of $A(\Phi)$, and $ R  = A(\Phi) - \Re(A( \Phi ))$. It is easy to see that the entries of the matrix $R$ are either zero or purely imaginary.  For every $ x\in \mathbb{R}^{n}$, $x^{T}Rx=0 $. Thus, for $ x\in \mathbb{R}^{n}$, we get $ x^{T}A(\Phi) x =x^{T}P x $. Now,
	\begin{align*}
		\lambda_{1}(P)& =\max_{x\in \mathbb{R}^{n}, ||x||=1}(x^{T}Px)\\
		&=\max_{x\in \mathbb{R}^{n}, ||x||=1}(x^{T}A(\Phi) x)\\
		&\leq \max_{x\in \mathbb{C}^{n}, ||z||=1}(z^{*}A(\Phi) z)\\
		&=\lambda_{1}(A(\Phi)).
	\end{align*}
	
	Suppose that $ \rho(A(\Phi))> 3\lambda_{1}(A(\Phi)) $. Then $ \lambda_{1}(P)\leq \lambda_{1}(A(\Phi))< \frac{1}{3} \rho(A(\Phi))$. Since $ \rho(A(\Phi))=\Big|\max\limits_{z\in \mathbb{C}^{n},||z||=1}(z^{*}A(\Phi) z)\Big|=|w^{*}A(\Phi) w|$, for some $ w\in \mathbb{C}^{n}$ with $ ||w||=1$, we have
	\begin{align*}
		\rho(A(\Phi))&=|w^{*}A(\Phi) w|\\
		&=|w^{*}P w+w^{*}R w|\\
		&\leq |w^{*}P w|+|w^{*}R w|  ~~~~~~[\mbox{By the triangle inequality}]\\
		&\leq |w^{*}R w|+\rho(P) \\
		&=|w^{*}R w|+\lambda_{1}(P) ~~~~~~[\mbox{By the Perron-Frobenius theorem}]\\
		&\leq |w^{*}R w|+\frac{1}{3}\rho(A(\Phi)).
	\end{align*}
	Thus $ |w^{*}R w|\geq\frac{2}{3}\rho(A(\Phi)).$
	Now, the matrix $ R $ can be written as $ R=iL $, where $ L $ is a skew-symmetric matrix with real entries. Therefore, the eigenvalues of $ L $ are either zero or purely imaginary and they occur in conjugate pairs. Hence the eigenvalues of $ R $ are  symmetric about the origin. Since $ R$ is a Hermitian matrix, so by the Theorem \ref{Theo.5}, the numerical range of $ R $ is $ [-\rho(R),\rho(R)] $. Therefore $ w^{*}R w \in [-\rho(R),\rho(R)] $ and hence there exists $ y \in \mathbb{C}^{n}$ with $ || y||=1 $ such that $ y^{*}Ry  = |w^{*}Rw|$. Now,
	
	\begin{align*}
		\lambda_{1}(A(\Phi))&=\max_{z\in \mathbb{C}^{n}, ||z||=1}(z^{*}A(\Phi) z)\\
		& \geq y^{*}A(\Phi) y\\
		&= y^{*} P y+y^{*} R y
	\end{align*}
	
	By the above observation, $y^{*}Ry  = |w^{*}Rw|\geq \frac{2}{3}\rho(A(\Phi)) $. Also, we have  $y^{*} P y \geq -\rho(P)$; and by the Perron-Frobenius theorem, $ \rho(P)=\lambda_1(P)< \frac{1}{3}\rho(A(\Phi)) $. Therefore,
	\begin{align*}
		\lambda_{1}(A(\Phi))\geq y^{*} P y+y^{*} R y \geq \frac{2}{3} \rho(A(\Phi))-\frac{1}{3} \rho(A(\Phi)) \geq \frac{1}{3} \rho (A(\Phi)),
	\end{align*}
	which is absurd. Thus,  we have  $\lambda_{1}(\Phi)\leq \rho(\Phi)\leq3\lambda_{1}(\Phi).$
\end{proof}

 For $ j=1,2 $, let $ G_{j}=( V(G_{j}), E(G_{j}))$ be two subgraphs of $ G $. Then the sum (or union) of these subgraphs $G_1$ and $G_2$ is denoted by $G_1 + G_2$, and is defined as the subgraph which consists of all the edges in $G_1$ or $G_2$ or both (see \cite[Page 12]{ore}). 

Let $G$ be a connected graph and  $T$ be a normal spanning tree of $G$. Let  $\{ C_{r_1}, C_{r_2}, \dots, C_{r_s}\}$ be the fundamental cycles associated with the normal spanning tree $ T $. Note that $|E(C_{r_j} \setminus (C_{r_1}+ C_{r_2}+ \dots + C_{r_{j-1}}))| \geq 1$ for $ j =2, 3, \dots, s$. 



\begin{definition}\label{f-graph}{\em
	A connected graph $ G $ is called an \textit{$ \mathcal{F} $-graph} if there is a sequence $ (C_{r_1}, C_{r_2}, \cdots, C_{r_s}) $ of all fundamental cycles of $ G $ with respect to some normal spanning tree $T$ of $ G $ which satisfy the following: $|E(C_{r_j} \setminus (C_{r_1}+ C_{r_2}+ \cdots + C_{r_{j-1}}))|>1$, for all $ j=2, 3, \dots, s $.}
\end{definition}

If  $ C_{1}, C_{2}, \dots, C_{s} $ are the fundamental  cycles of an $ \mathcal{F} $-graph $ G $, then, without
loss of generality, we assume that the sequence $ (C_{1}, C_{2}, \dots, C_{s}) $ satisfies the above property. 
%
%


The girth of an undirected graph is the length of a shortest cycle contained in the graph. 
\begin{prop}
	Let $ G $ be an $ \mathcal{F} $-graph of girth $ r $ and with $m$ edges and $ n $ vertices. Then  $ r\leq 2n-m.$
\end{prop}

\begin{proof}
	Let $ G $ be an $\mathcal{F} $-graph with $ m $ edges and $ n $ vertices. Let $T$ be a normal spanning tree of $ G $, and $ C_1, C_2, \cdots, C_{m-n+1} $ are the fundamental cycles of $ G $ with respect to $ T $ which have the property stated in Definition \ref{f-graph}. Since all the above cycles are fundamental cycles, so $ |E(C_j)\cap E(T) |=|E(C_j)|-1$. By the definition of $ \mathcal{F} $-graphs, $ |E(C_2)\setminus E(C_1) |\geq 2 $, so $|E(C_1+C_2)\cap E(T) | \geq |E(C_1)\cap E(T)|+1=|E(C_1)|$. Again $ |E(C_3)\setminus E(C_1+C_2) |\geq 2 $, so  $|E(C_1+C_2+C_3)\cap E(T) | \geq |E(C_1+C_2)\cap E(T)|+1\geq |E(C_1)|+1 $. By continuing this process, we get 
	$ |E(C_1+C_2+C_3+\cdots+C_{m-n+1})\cap E(T) | \geq |E(C_1)|+m-n-1$. Now, $ (n-1)=|E(T)|\geq|E(C_1+C_2+C_3+\cdots+C_{m-n+1})\cap E(T) |\geq r+(m-n-1)$ and hence $ r\leq 2n-m $.
\end{proof}

Let us give some examples and non-examples of $ \mathcal{F}$-graphs.
\begin{enumerate}
	\item[(1)] Any graph with vertex disjoint cycles is an $ \mathcal{F} $-graph.
	\item [(2)] If $ G $ is an $ \mathcal{F} $-graph, then every subgraph of $ G $ is also an $ \mathcal{F} $-graph.
	\item [(3)] A graph obtained by subdivision of every edge of a connected graph is an $ \mathcal{F} $-graph.
	\item[(4)] $ K_5 $ is not an $ \mathcal{F} $-graph. 
\end{enumerate}

Now, we present one of the main results of this section.

\begin{theorem}\label{Th-3.2.1}
	Let $ \Phi=(G, \varphi) $ be any $ \mathbb{T} $-gain graph on an $ \mathcal{F} $-graph $ G $ with $n$ vertices and $m$ edges. Then $$ \lambda_1(\Phi)\leq \rho(\Phi)\leq 3\lambda_1(\Phi).$$	
\end{theorem}
\begin{proof}	

 Let $ C_{1}, C_{2}, \dots, C_{m-n+1} $ be the fundamental cycles of $ G $ with respect to a normal spanning tree $T$, and the cycles statisfy the property given in Definition \ref{f-graph}. Let $ \varphi(\vec{C}_k)=e^{ir_k}$, for $ k=1,2, \dots, (m-n+1) $, where $ \vec{C}_k $'s are the oriented fundamental cycles in $ \vec{G}_T $. In the proof, we always scale $ r_k $ to the region $ (-\pi, \pi] $ for all $ k $. 
	
	\noindent \textbf{Claim:} There is a $ \mathbb{T} $-gain graph $ \Psi=(G, \psi) $ such that $ \Psi \sim \Phi $ and $ \Re(A(\Psi))\geq 0 $.
	
	\noindent Let us construct a $ \mathbb{T} $-gain graph $ \Psi=(G, \varphi) $ such that $ \psi(\vec{C}_k)=\varphi(\vec{C}_k)$ for all $ k $. Then $ \Psi \sim \Phi $. Now, define the gain function $\psi$ recursively as follows:  Define $ \theta_{1}=\frac{{r_{1}}}{|E(C_{1})|} \in [-\frac{\pi}{2}, \frac{\pi}{2}]$, where $ |E(C_{1})|$ is the number of edges of $ C_{1} $. For each edge   $\vec{e}$ of $ \vec{C}_{1} $, define $ \psi(\vec{e})=e^{i\theta_{1}}$. Then for each edge in $ E(\vec{C}_{2}) \setminus E(\vec{C}_{1}) $,  assign the gain as follows: Let $ |E(C_{1})\cap E(C_{2} )| =l_{1}$ and $ |E(C_{2})\setminus E(C_{1})|=k_{1} $. Consider $ \gamma_{1}= (r_{2}-\theta_{1}l_{1})~(\mod 2\pi)$ and scale it to the interval $ \left(-\pi, \pi \right] $. Let $ \theta_2=\frac{\gamma_{1}}{k_1} $. For each edge $ \vec{e} $ in $ E(\vec{C}_{2}) \setminus E(\vec{C}_{1}) $, define $ \psi(\vec{e})=e^{i\theta_2} $. Since $ G $ is an $ \mathcal{F} $-graph, so $ k_1\geq 2 $. Thus $ \theta_2 \in [-\frac{\pi}{2}, \frac{\pi}{2}] $. At the end of the process, we get $ e^{ i\theta_{1}}, e^{ i\theta_{2}}, \dots, e^{ i\theta_{m}}$ are the gains of the oriented edges of $G $, where $ \theta_{t}\in [-\frac{\pi}{2}, \frac{\pi}{2}]$ for $ t=1,2, \dots, m $. Since under this procedure the gains of the cycles $ \vec{C}_{1}, \vec{C}_{2}, \dots , \vec{C}_{m-n+1}$ remain unchanged, i.e., $ \varphi(\vec{C}_k)=\psi(\vec{C}_k)$, for $ k=1,2, \cdots, (m-n+1) $, so $ \Phi \sim \Psi $. Also $ \Re(\psi(\vec{e})) \geq 0$, for all edge $ e $ in $ G $. Thus $ \Re(A(\Psi)) \geq 0$.
	Now the result follows from Theorem \ref{main4.1}.
	
\end{proof}

\begin{corollary}
	Let $ G $ be a graph with cycles that are vertex disjoint. If $ \Phi=(G, \varphi) $ is any $ \mathbb{T} $-gain graph, then $ \lambda_{1}(\Phi)\leq \rho(\Phi) \leq 3 \lambda_{1}(\Phi)$.
\end{corollary}

\begin{corollary}
	Let $ G $ be a graph obtained by subdivision of each edge of some connected graph. Let $ \Phi=(G, \varphi) $ be any $ \mathbb{T} $-gain graphs on $ G $. Then $ \lambda_{1}(\Phi)\leq \rho(\Phi) \leq 3 \lambda_{1}(\Phi)$.
\end{corollary}

\section{Spectral radius and the largest vertex degree}\label{spec-bound}


It is well known that, if $ G $ is a simple connected graph with maximum vertex degree $\Delta$, then $ \rho(G) \leq \Delta$. Furthermore, equality holds if and only if $ G $ is regular. In this section, we first  extend this result for $\mathbb{T}$-gain graphs and characterize the extremal cases in terms of the gains defined on them. 

\begin{theorem}\label{Theo.4.5}
	Let $ \Phi $ be a $ \mathbb{T}$-gain graph on a connected graph $ G $ with the largest vertex degree $ \Delta $. Then $ \rho(\Phi)\leq\Delta $ and equality occurs if and only if $ G $ is $ \Delta $-regular and either $ \Phi $ or $ -\Phi $ is balanced.
\end{theorem}

\begin{proof}
	The proof follows from the above discussion, Theorem \ref{Theo.12} and Theorem \ref{Theo.9}.
\end{proof}


Our next goal is to study the structural properties of $\mathbb{T}$-gain graphs for which $ \rho(\Phi)= \Delta $ hold.  Unfortunately, we do not get a nice structure for the arbitrary $\mathbb{T}$-gain graphs. Nevertheless, for particular choices of the gains, we do get interesting structural characterization. Next, we introduce a generalization of Hermitian adjacency matrices of mixed graphs, for which the sharpness of the bounds in Theorem \ref{Theo.4.5} can be illustrated nicely. 

\begin{definition}
	Let $ X $ be a mixed graph and $k\in \mathbb{N} $. The $ k$\textit{-generalized Hermitian adjacency matrix} of $ X $ is denoted by $ H_{k}(X) $ and its $ (s,t)th $ entry is defined by $$h_{st}=\begin{cases}
		1& \text{if } \mbox{$e_{st} \in E(X)$},\\
 		e^{\frac{i\pi}{k+1}}& \text{if  } \mbox{$\overrightarrow{e_{st}} \in E(X)$ \text{and} $\overrightarrow{e_{ts}} \notin E(X)$},\\
		e^{-\frac{i\pi}{k+1}}& \text{if  } \mbox{$\overrightarrow{e_{st}} \notin E(X)$ \text{and} $\overrightarrow{e_{ts}} \in E(X)$},\\
		0&\text{otherwise.}\end{cases}$$
\end{definition}

It is easy to see that,  $ k$-generalized Hermitian adjacency matrices of mixed graphs are special case of $\mathbb{T}$-gain graphs with gains are from the set $\{1, e^\frac{i\pi}{k+1}, e^{-\frac{i\pi}{k+1}}\}.$
For each $ k\in \mathbb{N}$, we can associate  a unique matrix  $ H_{k}(X) $ with a given digraph $X$. If $ k=1 $, then $ H_1(X) $ is the Hermitian adjacency matrix of $X$. If $ k=2 $, then $ H_2(X) $ is the Hermitian adjacency matrix of second kind, studied in \cite{mohar-new}.

\vspace{6pt}
For a mixed graph $ X $, the underlying graph of $ X $ is denoted by $ \Gamma(X) $.
  Now we establish the main result of this section.

\vspace{7pt}
%
%
%

\begin{theorem}\label{Theo5.5}
	Let $ X $ be a mixed graph and $ \Delta $ be the largest vertex degree of $ \Gamma(X) $. Then $ \rho(H_{k}(X)) \leq \Delta $  for all $ k\in \mathbb{N} $. Furthermore, if $ X $ is weakly connected, then equality holds if and only if $ X $ is $ \Delta $-regular and there is a partition of the vertex set of $ X $ into $ (2k+2) $ parts (possibly empty) $ V_{0}, V_{\theta}, \dots, V_{(2k+1)\theta}$, where $ \theta=\frac{\pi}{k+1} $, such that $ X $ has one of the following structures (see Figure \ref{fig14}):
	\begin{enumerate}
		\item \textbf{Structure A:} Mixed graph induced by each vertex set $V_{m\theta} $, for each $ m\in \{0,1, \dots, (2k+1)\}$ contains only undirected edges. For each directed edge $ \overrightarrow{e_{st}} $ of $ X $, if $ v_{s}\in V_{p\theta } $, for some $ p \in \{0,1, \dots, (2k+1)\}$, then $ v_{t}\in V_{((p+1)\theta)~( \mod~2\pi)} $.
		\item \textbf{Structure B:} For every undirected edge $ e_{st} $ in $ X $, if $ v_{s} \in V_{m\theta}$ for some $ m $, then $ v_{t}\in V_{(\pi+m\theta)~( \mod~2\pi)}$.  That is, each vertex set $ V_{m\theta } $, for $  m\in \{0,1, \dots, (2k+1)\},$ is an independent set. Every directed edge $ \overrightarrow{e_{st}} $ in $ X $, if $ v_{s}\in V_{m\theta} $ for some $ m $, then $ v_{t}\in V_{(\pi+(1+m)\theta)~(\mod~2\pi)} $.
	\end{enumerate}
\end{theorem}
\begin{figure}[!htb]
	\begin{center}
		\includegraphics[scale=1.01]{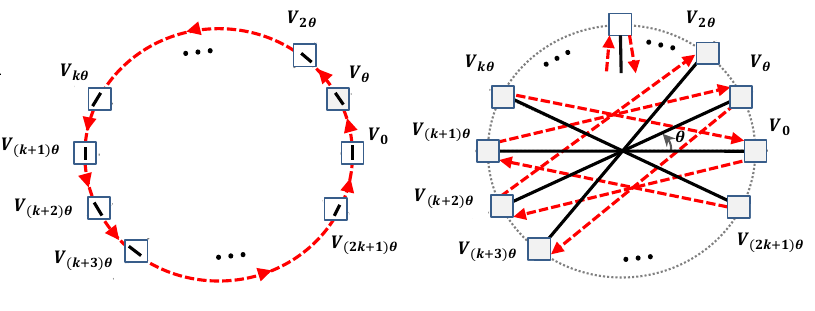}
		\caption{Structure A and Structure B for $ H_{k}(X) $ } \label{fig14}
	\end{center}
\end{figure}

In Figure \ref{fig14}, red dotted arrows and solid black lines denote directed and undirected edges, respectively.

\begin{proof}
	
	Let $ X $ be a mixed graph and $H_k(X)$ be the $k$-generalized Hermitian adjacency matrix of $X$ for some  $ k \in \mathbb{N}$. Let  $ G=\Gamma(X) $. Define $ \Phi=(G, \varphi)$ be a $ \mathbb{T} $-gain graph, where $ \varphi(\overrightarrow{e_{st}})$ is the $ (s,t) $-th entry of $ H_k(X) $. Therefore, by Theorem  \ref{Theo.12}, $\rho(H_{k}(X)) \leq \Delta $. Since $ X $ is weakly connected,  $ G $ is connected. Thus, by Theorem \ref{Theo.4.5}, equality holds if and only if $ G $ is $ \Delta$-regular and either $ \Phi $ is balanced or $ -\Phi $ is balanced. Let us show that $ \Phi $ is balanced if and only if $ X $ has the structure A, and $ -\Phi $ is balanced if and only if $ X $ has the structure B.  Let $ T $ be a normal spanning tree of $ G $ with the root vertex $ v_{1} $, and let  $ \theta=\frac{\pi}{k+1} $.
	
	\noindent{\bf Case (i):} Let  $\Phi$ be balanced. We partition the vertex set of $X$ as follows: 
	Let  $v_{1} \in V_{0}$. If a vertex $v_l$ is adjacent  to $v_1$ in $T$, then keep the vertex $v_l$ in one of the vertex partition as follows:
	\begin{enumerate}
		\item[(a)] $ v_{l}  \in  V_{\theta} $, if $ \varphi(\overrightarrow{e_{1l}}) = e^{i\theta} $,
		\item[(b)]   $ v_l \in V_{-\theta~(\mod ~ 2\pi)} ,$  if  $ \varphi(\overrightarrow{e_{1l}}) =  e^{-i\theta} $, and
		\item[(c)]  $ v_l \in V_{0} $ , if  $ \varphi(\overrightarrow{e_{1l}}) =  1. $
	\end{enumerate}
	
	In a similar way distribute the vertices of $V(G)$ into the sets $V_{0}, V_{\theta}, V_{2\theta}, \dots ,V_{(2k+1)\theta}$. Note that if $ v_s \in V_{p\theta} $ and $ v_t \in V_{q\theta} $, then $ \varphi(\overrightarrow{v_sTv_t})=e^{i(q-p)\theta} $, where $ 0\leq (q-p)\leq (2k+1) $. Let $ e_{st}\in E(G)\setminus E(T) $. Suppose $\overrightarrow{e_{st}}$ is a directed edge in $ X $ such that  $ v_s \in V_{p\theta} , v_t\in V_{q\theta} $, and $ \varphi(\overrightarrow{e_{st}})=e^{i\theta} $. Since  $ e_{st}\in E(G)\setminus E(T) $, there is a fundamental cycle $ C $ with respect to $ T $ such that $ e_{st}\in E(C) $.   Also $ \varphi(\overrightarrow{C})=\varphi(\overrightarrow{e_{st}})\varphi(\overrightarrow{v_tTv_s})=1$, as $ \Phi $ is balanced. Thus $\varphi(\overrightarrow{v_sTv_t})=e^{i\theta}$, and hence $q-p= 1$. A similar argument works for all the edges of $X$, hence $X$ has structure $A$.

	Conversely, let $ X $ have structure A. We will show that $ \Phi$ is balanced. In the proof, we  scale the labelling of vertex partitions into the region $ [0,2\pi)$.\\
	\textbf{Claim:} For any path $\overrightarrow{v_xPv_y}  $,  if $ v_x \in V_{g\theta} $ and $ v_y \in V_{h\theta} $, then $ \varphi(\overrightarrow{v_xPv_y})=e^{i(h-g)\theta} $.\\
	Notice that if $ v_s \in V_{g\theta} $ and $ v_s \sim v_t $, then $ v_t \in V_{(g+0)\theta} $ or $ v_t \in V_{(g+1)\theta} $ or $ v_t \in V_{(g-1)\theta} $ if $ \varphi(\overrightarrow{e_{st}})=1$ or $ e^{i\theta}$ or $e^{-i\theta}$, respectively. Suppose, $ \overrightarrow{v_xPv_y} $  has $ r $ undirected edges, $ p $ edges with gain $ e^{i\theta} $ and $ q $ edges with gain $ e^{-i\theta} $. Then $ \varphi(\overrightarrow{v_xPv_y})=e^{i(p-q)\theta} $. On the other hand, since $ v_x\in V_{g\theta} $, then $ v_y\in V_{(g+p-q)\theta} $. Therefore, $ g+p-q=h $ and hence $\varphi(\overrightarrow{v_xPv_y})=e^{i(h-g)\theta} $.
	
	Let $ C $ be a cycle and $ v_x, v_y $ be two vertices lie on $ C $. Then $\overrightarrow{C}$ can be written as $ (\overrightarrow{v_xP_1v_y})(\overrightarrow{v_yP_2v_x}) $, where  $ P_1 $ is the path between the vertices $v_x$ and $v_y$ in $\overrightarrow{C}$,  and $ P_2 $ is the path between the vertices $v_y$ and $v_x$ in $\overrightarrow{C}$. Then $\varphi(\overrightarrow{C})=\varphi(\overrightarrow{v_xP_1v_y})\varphi(\overrightarrow{v_yP_2v_x})$. Suppose $ v_x\in V_{g\theta} $ and $ v_y \in V_{h\theta} $. Then   $ \varphi(\overrightarrow{C})=\varphi(\overrightarrow{v_xP_1v_y})\varphi(\overrightarrow{v_yP_2v_x})=e^{i(h-g)\theta+i(g-h)\theta}=1$. Hence $ \Phi $ is balanced.\\
	
	\noindent {\bf Case (ii)}: Let $ -\Phi $ be balanced. We partition the vertex set of $X$ as follows: 
	Let  $v_{1} \in V_{0}$. If a vertex $v_l$ is adjacent  to $v_1$ in $T$, then keep the vertex $v_l$ in one of the vertex partition as follows:
	\begin{enumerate}
		\item[(a)] $ v_{l}  \in  V_{(\pi + \theta) (\mod ~ 2\pi)} $, if $ -\varphi(\overrightarrow{e_{1l}}) = -e^{i \theta} = e^{i(\pi + \theta)} $,
		\item[(b)]   $ v_l \in V_{(\pi -\theta)(\mod ~ 2\pi)} ,$  if  $ -\varphi(\overrightarrow{e_{1l}}) = - e^{-i\theta} = e^{i(\pi -\theta)} $, and
		\item[(c)]  $ v_l \in V_{(\pi +0)} $ , if  $ -\varphi(\overrightarrow{e_{1l}}) =  -1 =  e^{i\pi} . $
	\end{enumerate}
	Recursively distribute the vertices of $V(X)$ into the sets $V_{0}, V_{ \theta}, V_{2\theta}, \dots ,V_{(2k+1)\theta}$. Since $-\Phi$ is balanced, if $ e_{st}\in E(G)\setminus E(T) $, then, as in case(i), $- \varphi(\overrightarrow{v_{s}Tv_{t}})$ equals to   $-\varphi(\overrightarrow{e_{ts}})^{-1}$.  From the construction, it is clear that for every undirected edge $ e_{st} $ in $ X $, if $ v_{s} \in V_{m\theta}$ for some $ m $, then $ v_{t}\in V_{(\pi+m\theta)~( \mod~2\pi)}$. Similar to case(i), it follows that for every directed edge $ \overrightarrow{e_{st}} $ in $ X $, if $ v_{s}\in V_{m\theta} $ for some $ m $, then $ v_{t}\in V_{(\pi+(1+m)\theta)~( \mod~2\pi)} $. 
	
	The proof of the converse is similar to that of case (i).
\end{proof}

\begin{remark}
	By taking $ k=1 $ and $k =2 $ in  Theorem \ref{Theo5.5} we get the known results \cite[Theorem 5.1]{Bojan} and 	\cite[Theorem 4.1]{S_Li} , respectively. 
\end{remark}

\bigskip
\section*{Acknowledgments}
The authors thank Prof Thomas Zaslavsky, Binghamton University, for his comments and suggestions on an earlier version of the paper. Aniruddha Samanta thanks University Grants Commission(UGC)  for the financial support in the form of the Senior Research Fellowship (Ref.No:  19/06/2016(i)EU-V; Roll No. 423206). M. Rajesh Kannan would like to thank the SERB, Department of Science and Technology, India, for financial support through the projects MATRICS (MTR/2018/000986) and Early Career Research Award (ECR/2017/000643).
\bigskip
\section*{Data availability}
		Data sharing is not applicable to this article as no datasets were generated or analyzed during the current study.

\bibliographystyle{amsplain}
\bibliography{References}
\end{document}